%% file: root.tex
\documentclass[letterpaper, 10pt, conference,dvipsname]{ieeeconf}
\usepackage{cite}
\usepackage{amsmath,amssymb,amsfonts, amsthm,verbatim, mathtools,dsfont, bbold}
\usepackage{graphicx}
\usepackage{textcomp}
\let\labelindent\relax
\usepackage{enumitem}
\usepackage{diagbox}
\usepackage{times,bm}
\usepackage{booktabs}
\usepackage[hidelinks]{hyperref}
\allowdisplaybreaks[2]
\usepackage{caption}
\usepackage{subcaption}
\usepackage{svg}

\usepackage{algorithm, algpseudocode }
\usepackage{shortcuts}

\IEEEoverridecommandlockouts
\def\BibTeX{{\rm B\kern-.05em{\sc i\kern-.025em b}\kern-.08em
    T\kern-.1667em\lower.7ex\hbox{E}\kern-.125emX}}
\begin{document}
\input{arxiv_cover}
\title{Blameless and Optimal Control under Prioritized Safety Constraints}
\author{{Natalia Pavlasek$^{1}$, Sarah H.Q. Li$^{2}$, Beh{\c{c}}et A{\c{c}}{\i}kme{\c{s}}e$^1$, Meeko Oishi$^3$, and Claus Danielson$^3$}
\thanks{This work was supported by the National Science Foundation under Grant CMMI-2105631. Any opinions, findings, and conclusions or recommendations expressed in this material are those of the authors and do not necessarily reflect the views of the National Science Foundation.
}
\thanks{$^1$William E. Boeing  Department of Aeronautics and Astronautics, University of Washington at Seattle.
        {\tt\small pavlasek@uw.edu},
        {\tt\small behcet@uw.edu}, and
$^2$Department of Electrical Engineering and Information Technology, ETH Zürich, Zürich, Switzerland.
        {\tt\small sarahli@control.ee.ethz.ch}, and
$^3$University of New Mexico, 
        {\tt\small oishi@unm.edu}, and
        {\tt\small cdanielson@unm.edu}. 
        }
}

\maketitle

\begin{abstract}
In many resource-limited optimal control problems, multiple constraints may be enforced that are jointly infeasible due to external factors such as subsystem failures, unexpected disturbances, or fuel limitations. In this manuscript, we introduce the concept of \emph{blameless optimality} to characterize control actions that a) satisfy the highest prioritized and feasible constraints and b) remain optimal with respect to a mission objective. For a general optimal control problem with jointly infeasible constraints, we prove that a single optimization problem cannot find a blamelessly optimal control sequence. Instead, finding blamelessly optimal control actions requires sequentially solving at least two optimal control problems: one to determine the highest priority level of constraints that is feasible and another to determine the optimal control action with respect to these constraints. We apply our results to a rocket landing scenario in which violating at least one safety-induced landing constraint is unavoidable. Leveraging the concept of blameless optimality, we formulate blamelessly optimal controllers that can autonomously prioritize the constraints most critical to a mission.
\end{abstract}

\input{introduction}
\input{modelling}
\input{theorems}
\input{results}
\input{conclusion}


\bibliographystyle{IEEEtran}
\bibliography{reference}

\end{document}

%% file: arxiv_cover.tex
%
%
%
%
%
%
%
\def \myJournal {American Control Conference 2024}
\def \myDoi {}
\def \myPaperSiteName {IEEE Xplore}
\def \myPaperSiteLink {}
\def \myYear {2023}
\def \myPaperCitation{N. Pavlasek, S.H.Q. Li, B. A{\c{c}}{\i}kme{\c{s}}e, M. Oishi, and C. Danielson, ``Blameless and Optimal Control under Prioritized Safety Constraints,'' \url{http://arxiv.org/abs/2304.06625}, 2023.}


\begin{figure*}[t]

\thispagestyle{empty}
\begin{center}
\begin{minipage}{6in}
\centering
This paper has been submitted for presentation at \emph{\myJournal}. 
\vspace{1em}

This is the author's version of an article that has, or will be, published in this journal or conference. Changes were, or will be, made to this version by the publisher prior to publication.
\vspace{2em}




\vspace{15cm}
\copyright \myYear \hspace{4pt}IEEE. Personal use of this material is permitted. Permission from IEEE must be obtained for all other uses, in any current or future media, including reprinting/republishing this material for advertising or promotional purposes, creating new collective works, for resale or redistribution to servers or lists, or reuse of any copyrighted component of this work in other works.

\end{minipage}
\end{center}
\end{figure*}
\newpage
\clearpage
\pagenumbering{arabic} 

%% file: introduction.tex
\section{Introduction}
\label{sec:introduction}

Consider a scenario in which a planetary lander is not able to perform its primary landing and must instead autonomously select a landing site. In order to make the best use of resources and accomplish the greatest number of mission goals possible, the lander should select a site by evaluating the potential benefits, such as safety of the landing site, or proximity to sites of interest. Such an evaluation of trade-offs typically rests on an ordering of priorities: first protecting the lander from damage, then attempting to achieve the highest priority mission goals. Designing autonomous systems to adhere to a prioritization that reflects operational choices is an important yet largely unexplored problem~\cite{Wu2022_AIAA}, that reflects upon the autonomous systems' perceived reliability, trustworthiness, and overall effectiveness~\cite{Chancellor_2023}. A lander whose actions reflect mis-ordered priorities (i.e., selecting a landing site that results in damage to the vehicle at the cost of being close to a site of interest, for example) would be considered misguided and even blameworthy, as its actions are in conflict with priorities.  


In this paper, we propose the design of controllers for autonomous systems, that are both {\em optimal} and {\em blameless}.  We interpret blamelessness, based on a formal description in~\cite{Halpern_Kleiman-Weiner_2018}, as the ability to avoid actions whose outcomes are inconsistent with an ordered prioritization of constraints.  Such mis-prioritizations only appear when not all constraints can be satisfied, that is, when the system operates in regions of the state-space that are infeasible with respect to some constraints. For instance, in the 
landing scenario in Figure~\ref{fig:helicopter}, if a landing zone that satisfies all safety constraints exists, an alternative landing site does not have to be considered.

\begin{figure}
    \centering
    \includegraphics[width=\columnwidth]{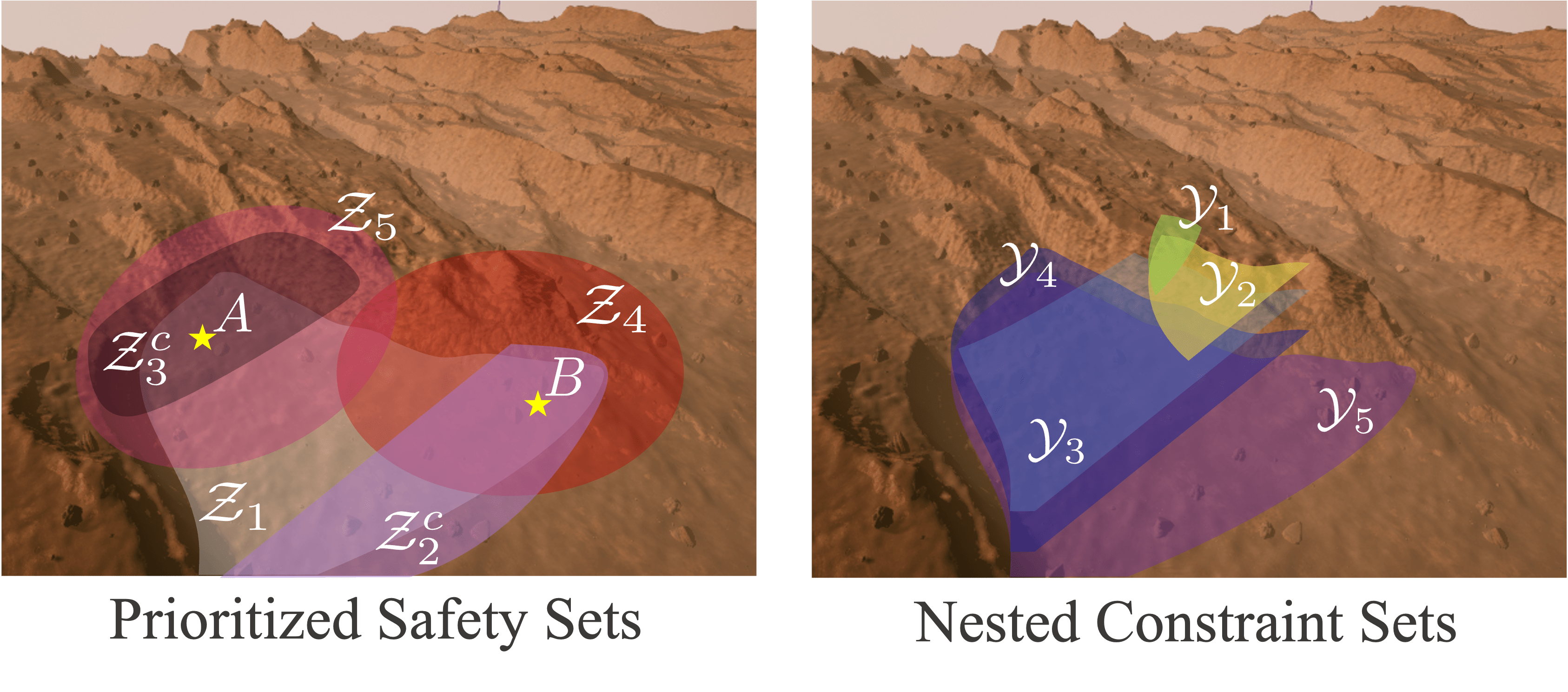}
    \caption{Optimal controllers for a lander should satisfy constraints in order of their priority, in order to be {\em blameless}.}
    \label{fig:helicopter}
\end{figure}

While extensive work has considered finding optimal solutions under constraints, considerably less work has been done on designing solutions under {\em infeasible} constraints. Unlike standard optimal control problems, in which infeasibility is commonly averted by softening the constraints~\cite{Weiss2015_MPC_soft_constraints}, controller design under infeasibility must focus not only on satisfying as many constraints as possible, but also on satisfying higher priority constraints over ones that are lower priority. That is, ensuring blamelessness in autonomous systems means that controllers must be designed to satisfy constraints in a manner that reflects their underlying prioritization.  Further, due to the urgency associated with operation in infeasible scenarios, methods to synthesize blameless controllers should have low computational cost. 

\textbf{Related work.}
Lexicographic optimization, in which a finite number of ordered objective functions are sequentially optimized, such that low priority objectives do not interfere with higher priority objectives~\cite{Isermann_linlexico}, is one method to ensure prioritization amongst constraints.  It has been used to address applications of prioritized safety~\cite{Lesser_2017, Dueri_2015}, as well as to relax constraints in decreasing order of priority under infeasibility~\cite{Vada2001_MPCinfeasibility, Miksch2011_lexicographicMPC}.  
However, the primary disadvantage of lexicographic methods is that even for convex problems, the problem must be solved iteratively, leading to a high computational burden. Typically lexicographic optimization requires solving as many optimization problems as there are prioritized objectives~\cite{Marques_boollexico}.  

Related work in reachability has examined conditions under which a controller exists, that ensures feasibility with respect to a known set of constraints~\cite{Tomlin2000, Bansal2017_reachability}.  However, these approaches typically presume that all constraints are of the same priority, or rely upon a pre-determined trade-off between safety and performance~\cite{Vinod_2018}. Reachability and positive invariance have been integrated into path planning approaches that prioritize safety~\cite{Starek2017_FMT,Leung2020_reachability, Vinod_Rice_2018, Chiang2015, Danielson2016, Dixit2018_MPCpotentialreach}, but are not readily extendable to multiple priorities and objectives.

\textbf{Contributions}.  
We provide a formal definition of blamelessness and formulate the problem of designing blameless controllers within an optimal control framework, amenable to a wide variety of application domains. We show that for general objective functions and constraints, it is not possible to solve for optimality and blamelessness in a single optimization problem, i.e., that no single continuous objective function can produce a blamelessly optimal action. We further show that it is possible to solve for optimality and blamelessness in exactly two optimization problems; the first determines the highest priority set of constraints that can be solved and the second finds the optimal control actions under these constraints. 
We demonstrate our algorithm on a real-time rocket landing problem, whose safety requirements are jointly infeasible due to fuel limitations.

\subsubsection*{Notation}
We use the shorthand $x_{0:N} $ to denote the sequence of variables $x_{0},\ldots, x_{N}$, for $x_i \in \mathbb R^n, i \in \{1, \cdots, N\}$. A sequence of constrained control inputs is denoted as $u_{0:N-1} = u_{0}, \ldots, u_{N-1}, u_{k}\in \Uset \subseteq \mathbb R^\ell$, with $u_{0:N-1} \in \Uset^N$. The notation $[N]$ is used to denote $\{1,\ldots,N\}$. For the set $\mc{X}$, the complement is denoted $\mc{X}^c$, so that $x\in\mc{X}^c$ implies $x\notin\mc{X}$.

%% file: modelling.tex
\section{Blameless Optimal Control} \label{sec:modelling}
This section introduces the concept of a blamelessly optimal control sequence given user-prioritized constraints. 
The state trajectory $x_{0:N}$ is given by the discrete-time dynamics,
\begin{align}\label{eqn:f_dynamics}
    x_{k+1 }:= f(x_{k }, u_{k }) \in \reals^n, \forall k \in [N-1],
\end{align}
under control sequence $u_{0:N-1} \in \reals^{N\ell}$ and initial state $x_0 \in \Xset_0 \subseteq \reals^n$.
We assume that the dynamics~\eqref{eqn:f_dynamics} are subject to a set of \emph{prioritized constraints} that are ranked by importance. 

\begin{assumption}[Prioritized safety sets]\label{assum:safetysets}
The user-defined compact sets $\{\mc{Z}_i\}_{1\leq i\leq m} = \mc{Z}_1, \ldots, \mc{Z}_m \subseteq \mc{R}^p$, with $p\leq n$, are prioritized state constraints on a subset of the state $x_k$, such that $\mc{Z}_1$ has the highest priority. For ease of notation, it is assumed that $p = n$, meaning that the full state at a given time is constrained by the safety sets.
\end{assumption}

The ordering of the sets $\Zset_i$ is user-defined and dictates the prioritization of the constraints to be satisfied in the event that the system dynamics and control constraints cause the problem subject to a subset of the safety constraints to be infeasible.
Also note that, without loss of generality, we presume that these constraint sets are terminal constraints, i.e., applicable to only the state at the last time instant, $x_N$.
We recast the prioritized safety sets as nested constraint sets.

\begin{definition}[Nested constraint sets]\label{def:Yi}
Let the nested constraint set $\Yset_i$ be defined by $\Yset_i = \bigcap_{j=1}^{m+1-i} \mathcal{Z}_j$, such that $\Yset_1\subseteq\ldots\subseteq\Yset_m$. 
\end{definition}

\begin{example}
Consider a sample-return mission in which an autonomous lander is carrying a rover that will drive from the landing site to sites of interest to collect samples, but can only drive a limited distance. A safe landing area is defined near two points of interest, denoted by $A$ and $B$. If feasible, the lander should select a landing site from which the rover can collect material from both sites of interest, but should avoid landing too close to the sites and risking contaminating the material. This example is depicted in Figure~\ref{fig:helicopter}. The prioritized sets represent first landing in the designated landing area ($\mathcal{Z}_1$), then avoiding regions which may damage the scientifically interesting material $B$ and $A$ ($\mathcal{Z}_2$ and $\mathcal{Z}_3$, respectively), then landing in an area from which the interesting sites $B$ and $A$ can be reached by the rover ($\mathcal{Z}_4$ and $\mathcal{Z}_5$, respectively). The nested constraint sets describe landing sites that will not contaminate sites $A$ or $B$ and from which the rover can reach both $A$ and $B$ ($\Yset_1$), sites that will not contaminate sites $A$ or $B$ and from which the rover can only reach site $B$ ($\Yset_2$), sites that will not contaminate sites $A$ or $B$, but from which the rover cannot reach either site ($\Yset_3$), and sites that avoid contaminating only site $B$ ($\Yset_4$), and all sites in the designated landing area ($\Yset_5$).

\end{example}

Nested constraint sets provide an intuitive understanding of blameworthiness, in which it is desirable for an autonomous system to sacrifice satisfaction of low priority constraints to ensure the satisfaction of high priority constraints. 
For instance, we are willing to sacrifice reaching site $A$ if it ensures that the lander does not damage the area surrounding site $B$.
In short, we wish to maximize the index $i$ such that $\Zset_j$ is satisfied for all $j \leq i$, or equivalently, minimize $i$ such that $\Yset_i$ is satisfied.

\begin{remark}
We assume that at least one of the nested constraint sets is non-empty $\Yset_i \neq \varnothing$. Thus, $\Yset_j \supseteq \Yset_i$ are non-empty for $j \in \{i,\dots,m\}$. 
\end{remark}

We formally define the concept of \emph{blameworthiness} of a control sequence with respect to the nested constraint sets.
\begin{definition}[Blameworthy control sequence]\label{def:blameworthy}
Suppose that the smallest nested constraint set that can be satisfied given the system dynamics, control constraints and initial condition is $\Yset_{i^\star}$. A control sequence $u_{0:N-1} \in \reals^{N\ell}$ that results in the state trajectory $x_{0:N}$ is \emph{blameworthy} if $x_N \notin \Yset_{i^\star}$.
\end{definition}

A control sequence is blameworthy if there exists an alternative control sequence that reaches a higher priority safety set. In Example 1, a control sequence that causes the lander to land in a region that may damage sites $A$ or $B$ is blameworthy if there is an alternative control sequence that would cause the lander to land in a region in which damage is unlikely to occur. If there exists no such alternative control sequence, it is blameless.

\begin{definition}[Blameless control sequence]\label{def:blameless}
A control sequence $u_{0:N-1} \in \reals^{N\ell}$ is \emph{blameless} if it is not blameworthy.
\end{definition}

A blameless control sequence minimizes the index $i\in [m]$ such that the resulting state sequence $x_{0:N}$ given by~\eqref{eqn:f_dynamics}, satisfies $x_N \in \Yset_i$.

For the set of initial states given by $\mc{X}_0 \subseteq \reals^n$ and the input constraint set by $\mc{U} \subseteq \reals^{\ell}$, we define the dynamically feasible set, defined by the dynamics~\eqref{eqn:f_dynamics} and the control constraints, $u_{0:N-1} \in \Uset^N$ as follows.

\begin{definition}[{Dynamically feasible set}]
The set of states and control sequences achievable from the initial state subject to the dynamics and control constraints is called the \emph{dynamically feasible set}. It is denoted as
\begin{multline}\label{eqn:opt_ctrl_feasible_set}
    \mc{F}(\mc{X}_0, \mc{U}) = \{(x_{0:N}, u_{0:N-1}) \ | \  x_{k+1 } = f(x_{k }, u_{k }), \\
    u_k \in \mc{U}, x_0 \in \mc{X}_0, \forall  k\in [N\!-\!1]\} \subseteq \reals^{(N+1)n + N\ell}.
\end{multline}
\end{definition}
The concepts of blameworthiness and blamelessness are depicted in Figure~\ref{fig:blameworthy_blameless}. The set $\mc{F}_N(\mc{X}_0, \mc{U}) = \{x_N | x_N \in f \textrm{ for some } f \in \mc{F}(\mc{X}_0, \mc{U})\}$ is the set of terminal states in the dynamically feasible set. A controller that results in a terminal state that lies in the intersection of $\mc{F}_N(\mc{X}_0, \mc{U})$ and $\mc{Y}_3$ is blameless since $\mc{Y}_3$ is the highest priority set that is dynamically feasible. A control sequence that results in a terminal state that is not in $\mc{Y}_3$ is blameworthy since a solution that results in the terminal state being in a higher priority set exists.

\begin{figure}
    \centering
    \includegraphics[width=\columnwidth]{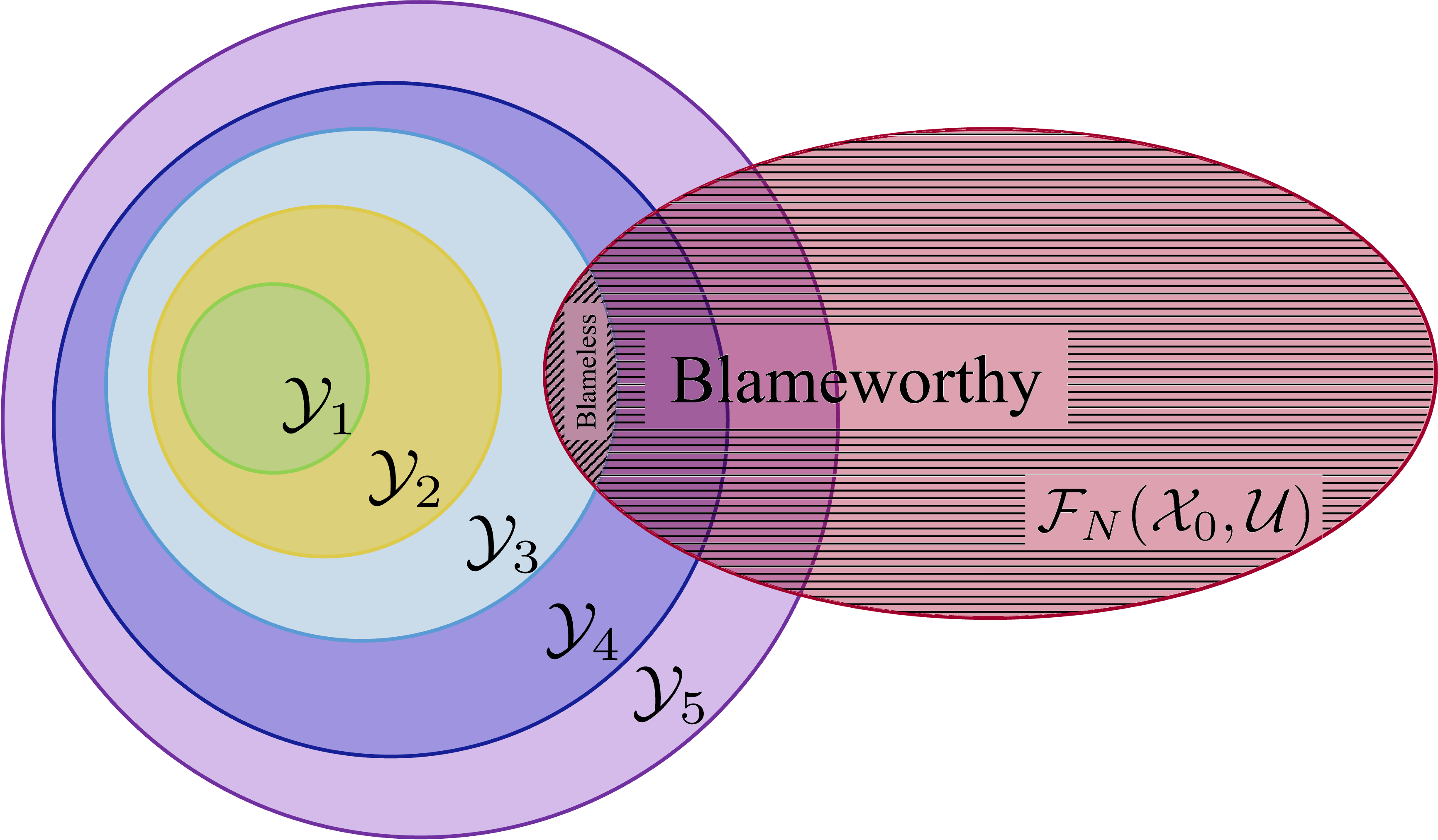}
    \caption{Illustration of the concepts of blameworthiness and blamelessness. }
    \label{fig:blameworthy_blameless}
\end{figure}

We assume the user-defined continuous objective, referred to as the \emph{mission objective}, is
\begin{align}
    q(x_{0:N}, u_{0:N-1}):\reals^{Nn} \times \reals^{N\ell} \mapsto \reals,\label{eq:objective}
\end{align}
where $q$ evaluates the cost of each state and control sequence. For a given objective $q$, we can define a blamelessly optimal control sequence as follows. 

\begin{definition}
[Blameless optimality]\label{def:blameless_optimal}
Consider a state and control sequence $(x_{0:N}, u_{0:N-1})\in \mc{F}(\mc{X}_0, \mc{U})$, where $\mc{F}(\mc{X}_0, \mc{U})$ is given by~\eqref{eqn:opt_ctrl_feasible_set}.
The control sequence $u_{0:N-1}$  is \emph{blamelessly optimal} if 
\begin{enumerate}
    \item It is blameless according to Definition~\ref{def:blameless}, and
    \item For all $(\hat{x}_{0:N}, \hat{u}_{0:N-1}) \in \mc{F}(\mc{X}_0, \mc{U})$ where $\hat{u}_{0:N-1}$ is blameless, 
\begin{align}\label{eq:optimal}
    q(x_{0:N}, u_{0:N-1}) \leq q(\hat{x}_{0:N}, \hat{u}_{0:N-1}).
\end{align}
\end{enumerate}
\end{definition}

\begin{problem} \label{prob:Pi}
Given nested constraint sets $\{\Yset_i\}_{1\leq i\leq m}$ (Definition~\ref{def:Yi}), and initial state $x_0\in\Xset_0$, find a blamelessly optimal control sequence $u_{0:N-1} \in \mc{U}^{N}$.
\end{problem}

Problem~\ref{prob:Pi} entails finding a control sequence that leads to the satisfaction of the largest number of safety sets possible, while taking into account the prioritization of the sets and optimality with respect to objective $q$ in \eqref{eq:objective}.

\begin{remark}
    Critically, we
    assume that not all nested constraint sets $\mc{Y}_i$ are feasible. Therefore, Problem~\ref{prob:Pi} is equivalent to finding 1) the smallest index $i^\star$ such that $x_N \in \mc{Y}_{i^\star}$ is feasible, and 2) the control sequence $u_{0:N-1}$ that is optimal under the constraint $x_N \in \mc{Y}_{i^\star}$. 
\end{remark}


%% file: theorems.tex
\section{Solving for Blamelessly Optimal Control Sequences}
\label{sec:theorems}

Blamelessness and optimality may be competing objectives when the optimal solution with respect to objective $q$ lies outside of the highest priority safety set, as is the case in Figure~\ref{fig:tradeoff_example}. In this section, we discuss various solution methods for finding blamelessly optimal control sequences.

\subsection{Need for Blameless Optimality}
We illustrate the need for blameless optimality in Figure~\ref{fig:tradeoff_example} for a one-dimensional state space and a one-dimensional control space. The objective $q$ given by~\eqref{eq:objective} is a quadratic function of the state, shown by the blue curve. The nested safety sets are by the shaded regions: the green, yellow, and blue regions correspond to $\Yset_1$, $\Yset_2$, and $\Yset_3$, respectively. The red hashed region shows states that are infeasible under $\mc{F}(\mc{X}_0, \mc{U})$ given by~\eqref{eqn:opt_ctrl_feasible_set}. Although $x_{\mathrm{blmless-opt}}$ has a higher cost than $x_{\mathrm{opt}}$, it is preferable to $x_{\mathrm{opt}}$ because it satisfies a higher priority nested constraint set, $\Yset_1$.

\begin{figure}
    \centering
    \includegraphics[trim={0cm 2cm 0 3cm},clip, width=\columnwidth]{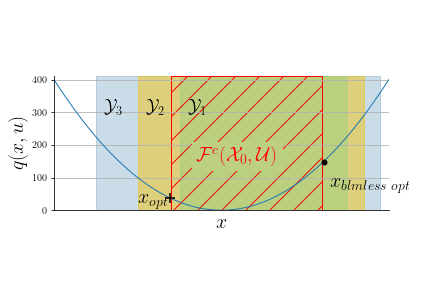}
    \caption{Visualization of the need for blameless optimality. Point $x_{opt}$ is optimal with respect to the blue objective subject to the dynamics constraints but is in $\Yset_2$. Point $x_{blmless\;opt}$ achieves a higher cost than $x_{opt}$, but is blamelessly optimal since is in $\Yset_1$.}
    \label{fig:tradeoff_example}
\end{figure}

\subsection{Related Methods}
The concept of priority within optimization problems is not novel. Some related concepts the exist in the literature are discussed in the following sections.

\subsubsection{Connection with Lexicographic Optimization}

The lexicographic optimization problem 
\begin{align*}
    \min_{w \in \mc{W}} \; \left( q_1(w), \ldots, q_m(w) \right),
\end{align*}
in which $w \in \reals^\eta$ is the solution variable, $\mc{W}$ is the set of feasible solutions, and $\left( \cdot, \ldots, \cdot \right)$ denotes an ordering, is typically addressed by iteratively solving
\begin{align*}
    &\mathrm{for}~j=1,\ldots, m:\\
    &\hspace{0.5cm}\min_{w \in \mc{W}}\; q_j(w) \quad \st\; q_i \leq q_i^\star,\quad i = 1,\ldots, j-1,
\end{align*}
where $q_i^\star$ is the optimal value of the $i^{th}$ problem. 
In lexicographic optimization, multiple ordered objectives are optimized. This is related to Problem~\ref{prob:Pi}, in that a notion of priority exists within an optimization problem.

It is possible to construct an algorithm that uses an iterative approach similar to that of lexicographic optimization to solve for a blamelessly optimal control sequence. This algorithm, shown below, iteratively imposes constraints in order of priority.

\begin{algorithm}
\caption{Brute-Force Blameless Control}
\label{alg:brute}
\begin{algorithmic}
\State Set $i=0$
\While{not feasible}
\State Loosen constraints $i\gets i+1$
\State Solve \vspace{-1em}
\begin{subequations}
\label{eq:brute}
\begin{alignat}{3}
&& \min_{u_{0:N-1}} ~& q(x_{0:N},u_{0:N-1}) \\
&& \st ~& x_{k+1 } = f(x_{k },u_{k }), \;x_{0 } \in \Xset_0, \\
&&& u_{k } \in \Uset, \;x_{N } \in \Yset_i 
\end{alignat}
\end{subequations}
\EndWhile
\end{algorithmic}
\end{algorithm}

\begin{proposition}
Under Assumption~\ref{assum:safetysets}, Algorithm~\ref{assum:safetysets} produces a control sequence $u_{0:N-1} \in \mc{U}^N$ that is blamelessly optimal (Definition~\ref{def:blameless_optimal}).
\end{proposition}

\begin{proof}
Suppose the control sequence found using Algorithm~\ref{alg:brute} was blameworthy for constraint $\Yset_j$. Then, by definition there exists $\hat{u}_{0:N-1 } \in \Uset^N$ such that $\hat{x}_{N } \in \Yset_j$. Thus, problem~\eqref{eq:brute} is feasible and therefore $x_{N } \in \Yset_j$.
\end{proof}

While Algorithm~\ref{alg:brute} produces a blamelessly optimal control sequence, it is not an ideal solution because of the excessive cost that comes from solving as many optimization problems as there are prioritized constraints. For a scenario with $m$ prioritized safety constraints, Algorithm~\ref{alg:brute} solves $m-1$ infeasible optimization problems and $1$ feasible optimization problem in the worst case. We are interested in finding a more computationally efficient solution to Problem~\ref{prob:Pi}, that is amenable to real-time application in safety-critical systems.

\subsubsection{Connection with Reachability Analysis}

We can show that blamelessness can equivalently be defined using the successor sets used in reachability analysis~\cite{Althoff_reachability_successor}. A successor set is defined as follows.

\begin{definition}[Successor Set]\cite[Def.10.3]{borrelli2017predictive} Consider the set of initial conditions $\mc{X}_0 \subseteq \reals^n$ and set of inputs $\Uset$. The $N$-step successor set under the input constraints $u_{0:N-1} \in \mc{U}^N$ and dynamics~\eqref{eqn:f_dynamics} is given by 
\begin{multline}
    \suc{\mc{X}_0, \mc{U}^N} = \{x_1, \ldots, x_N \ | \ \exists u_k \in \mc{U}, x_0 \in \mc{X}_0 \\  x_{k+1} = f(x_k, u_k), \forall k \in [N] \}.
\end{multline}

\end{definition}
We use the shorthand $\suc{\mc{X}_0} = \suc{\mc{X}_0,\mc{U}^N}$, and denote $n^{th}$ state in $\suc{\mc{X}_0}$ as $\text{Suc}_n(\mc{X}_0)$. For a given input sequence $u_{0:N-1} \in \Uset^{N}$, we use the shorthand $\suc{\mc{X}_0, u_{0:N-1}}$. 
The following proposition formulates blamelessness using successor sets.

\begin{proposition}
\label{prop:reach}
The control sequence $u_{0:N-1} \in \mc{U}^N$ is blameless if $\Yset_i \cap \text{Suc}_N(\mc{X}_0) \neq \varnothing$ implies
$\text{Suc}_N(\mc{X}_0, u_{0:N-1}) \in \Yset_i$ for all $i= 1,\dots,m$.
\end{proposition}

\begin{proof}
    If $\Yset_i \cap \text{Suc}_N(\mc{X}_0) \neq \varnothing$, then there is a control sequence $u_{0:N-1}$ such that $x_N \in \Yset_i$. Then, $u_{0:N-1}$ is blameless if $\text{Suc}_N(\mc{X}_0, u_{0:N-1}) \in \Yset_i$.
\end{proof}

\subsection{Blamelessness and Optimality Relationship}

In this section, we show that in general, finding blamelessly optimal control sequences requires the solution of two optimization. We show this by contradiction. We use this result to formulate a framework for solving for blamelessly optimal control sequences that moves the computational complexity involved in the brute force method,~\eqref{eq:brute}, offline.

Consider the problem of designing an objective function that produces control sequences that are simultaneously blameless and optimal with respect to objective $q$. We introduce the notation $\hat{q}$ to represent a continuous objective that, when minimized, results in a control sequence that is blameless.

\begin{problem}\label{prob:cts_obj}
Find a continuous objective function $\hat{q}:\reals^{Nn}\times \reals^{N\ell}\mapsto \reals$ for the optimization problem
\begin{subequations}
\label{eqn:cts_obj}
\begin{alignat}{3}
&&\min_{u_{0:N-1} \in \reals^{N\ell}}~&  \hat{q}(x_{0:N},u_{0:N-1}) \label{eqn:cts_obj-cost} \\
&&\st~& (x_{0:N}, u_{0:N-1}) \in \mc{F}(\mc{X}_0, \mc{U}), \label{eqn:cts_obj-constraints}
\end{alignat}    
\end{subequations}
whose solution $u_{0:N-1}^\star$ is blameless, and optimal with respect to the user-defined objective $q$.
\end{problem}

We define the following set to facilitate the solution to Problem~\ref{prob:cts_obj}. Consider all the dynamically feasible state and control sequences, $(x_{0:N},{u_{0:N-1}) \in \mc{F}(\mc{X}_0, \mc{U})}$, that have a cost $\hat{q}(x_{0:N}, u_{0:N-1}) \leq \alpha_i$ for $i=1,\ldots,m$. We denote the set of terminal states attained by these state and control sequences as
\begin{multline}\label{eqn:levelset}
\mc{H}_{i}(\hat{q}) = \{ x_{N} \mid (x_{0:N}, u_{0:N-1}) \in \mc{F}(\mc{X}_0, \mc{U}),\\
\hat{q}(x_{0:N}, u_{0:N-1}) \leq \alpha_i \}.
\end{multline}
The set $\mc{H}_i(\hat{q})$ is the set of terminal conditions in the $\alpha_i$ sublevel set of $\hat{q}$.

The following theorem presents the necessary and sufficient conditions for designing an objective $\hat{q}$ that produces blameless control sequences.

\begin{theorem} \label{thm:optimal_blameless}
Under the nested constraint sets $\{\mc{Y}_i\}_{1\leq i\leq m}$ defined in Definition~\ref{def:Yi}, the optimal control problem~\eqref{eqn:cts_obj} produces blameless control sequences if and only if the continuous objective $\hat{q}$ satisfies
\begin{equation}
\label{eq:optimal_blameless}
    \mc{H}_i(\hat{q}) = \mc{Y}_i \cap \text{Suc}_N(\mc{X}_0), \ \forall \ 1\leq i\leq m.
\end{equation}
\end{theorem}

\begin{proof}
Note that $\mc{H}_{i}(\hat{q})$ are compact sets since the dynamics~\eqref{eqn:f_dynamics} and the objective function $\hat{q}(x_{0:N}, u_{0:N-1})$ are continuous functions and $\Uset$ is a compact set.

First, we prove by contradiction that if $\hat{q}$ satisfies~\eqref{eq:optimal_blameless}, then~\eqref{eqn:cts_obj} produces blameless control sequences. Suppose $\hat{q}$ satisfies~\eqref{eq:optimal_blameless}, but the result of~\eqref{eqn:cts_obj} is blameworthy. This means if the solution to $\hat{q}$ is $({x}_{0:N}^\star, {u}_{0:N-1}^\star) \in \mc{F}(\mc{X}_0, \mc{U})$ with $x_N^\star \not \in \Yset_i$, there exists another state and control sequence pair, $({x}^\dag_{0:N}, {u}^\dag_{0:N-1}) \in \mc{F}(\mc{X}_0, \mc{U})$ with ${x}^\dag_N \in \Yset_i$. Then, by~\eqref{eq:optimal_blameless}, we have $\hat{q}({x}^\dag_{0:N}, {u}^\dag_{0:N-1}) \leq \alpha_i < \hat{q}({x}_{0:N}^\star, {u}_{0:N-1}^\star)$, which contradicts the definition of an optimal solution of~\eqref{eqn:cts_obj}. 

We will prove the reverse implication directly. Suppose~\eqref{eqn:cts_obj} produces blameless control sequences. Then, by construction,
\begin{align*}
    \begin{aligned}[t]
    \max_{u_{0:N-1}} ~&~ \hat{q}(x_{0:N},u_{0:N-1}) \\ 
    \text{s.t. } &  \mc{F}(\mc{X}_0, \mc{U}) \\
    & x_N \in \Yset_i 
    \end{aligned}
    \leq \alpha_i = 
    \begin{aligned}[t]
    \inf_{u_{0:N-1}} ~&~ \hat{q}(x_{0:N},u_{0:N-1}) \\ 
    \text{s.t. } & \mc{F}(\mc{X}_0, \mc{U}) \\
    & x_N \notin \Yset_i
    \end{aligned}
\end{align*}
since, otherwise, there exists a solution to~\eqref{eqn:cts_obj}, $({x}_{0:N}^\star,{u}_{0:N-1}^\star) \in \mc{F}(\mc{X}_0, \mc{U})$ with $x_N^\star \notin \Yset_i$, for which $\hat{q}({x}_{0:N}^\star,{u}_{0:N-1}^\star) \leq \hat{q}({x}_{0:N}^\dag,{u}_{0:N-1}^\dag)$ for some $({x}_{0:N}^\dag,{u}_{0:N-1}^\dag) \in \mc{F}(\mc{X}_0, \mc{U})$ with ${x}_N^\dag \in\Yset_i$. That is, the control sequence $u_{0:N-1}^\star$ found by solving~\eqref{eqn:cts_obj} is blameworthy, which results in a contradiction. Thus, if $\hat{q}(x_{0:N},u_{0:N-1})\leq \alpha_i$ then $x_N \in \Yset_i$, and we have $x_N \in \mc{H}_{i}(\hat{q})$ implies $x_N \in \Yset_i$, \ie~$\mc{H}_{i} (\hat{q}) \subseteq \Yset_i \cap \text{Suc}_N(\mc{X}_0)$.
Conversely, if $z \in \Yset_i \cap \text{Suc}_N(\mc{X}_0)$ then there exists $(x_{0:N}, u_{0:N-1}) \in \mc{F}(\mc{X}_0, \mc{U})$ with $x_N = z$. By definition of $\alpha_i$, we have $\hat{q}(x_{0:N},u_{0:N-1}) \leq \alpha_i$. Thus, $z \in \mc{H}_i(\hat{q})$ i.e. $\Yset_i \cap \text{Suc}_N(\mc{X}_0) \subseteq \mc{H}_{i} (\hat{q})$.

\end{proof}

Theorem~\ref{thm:optimal_blameless} provides the necessary and sufficient conditions~\eqref{eq:optimal_blameless} for constructing an objective function $\hat{q}$ for which~\eqref{eqn:cts_obj} produces blameless control sequences. 
Next, we consider whether it is possible to construct the objective function $\hat{q}$ such that it produces control sequences that are both blameless and optimal with respect to the original user-defined objective function~$q$. The following corollary shows that there are continuous objectives $q$ for which all continuous objectives that produce blameless control sequences, $\hat{q}$ produce sub-optimal control sequences. That is, there does not exist an objective satisfying~\eqref{eq:optimal_blameless} that produces control sequences that are optimal with respect to $q$.

\begin{corollary} \label{cor:optimal_blameless}

There exists a continuous objective $q$ such that there is no objective $\hat{q}$ that produces blameless control sequences that are also optimal with respect to $q$.

\end{corollary}

\begin{proof}
We will prove by construction that there exists an objective $q$ whose optimal solution does not correspond with the optimal of any objective $\hat{q}$ satisfying~\eqref{eq:optimal_blameless}.

Consider a continuous objective $\hat{q}({x}_{0:N}, {u}_{0:N-1})$ that satisfies~\eqref{eq:optimal_blameless}. We will show $\hat{q}({x}_{0:N}, {u}_{0:N-1}) = \alpha_i$ for all ${x}_N$ on the boundary of $\Yset_i$, denoted $\delta \Yset_i$. If $\hat{q}({x}_{0:N}, {u}_{0:N-1}) > \alpha_i$ for $x_N \in \delta \Yset_i$, then by continuity of $\hat{q}$, we have $H_i \not\subseteq \Yset_i$, and likewise if $\hat{q}({x}_{0:N}, {u}_{0:N-1}) < \alpha_i$ for $x_N \in \delta \Yset_i$ then by continuity, $\Yset_i \not\subseteq H_i$. Thus, by contradiction with~\eqref{eqn:levelset}, we have $\hat{q}({x}_{0:N}, {u}_{0:N-1}) = \alpha_i$ for all $x_N \in \delta \Yset_i$. Furthermore, by~\eqref{eqn:levelset}, for any $({x}_{0:N}^\dag, {u}_{0:N-1}^\dag)\in \mc{F}(\Xset_0, \Uset)$ with  ${x}_N^\dag$ in the interior of $\Yset_i$ and $({x}^\ddag_{0:N}, {u}_{0:N-1}^\ddag)\in\mc{F}(\Xset_0, \Uset)$ with ${x}_N^\ddag \in \delta\Yset_i$, we have $\hat{q}({x}_{0:N}^\dag, {u}_{0:N-1}^\dag) \leq \hat{q}({x}_{0:N}^\ddag, {u}_{0:N-1}^\ddag)$.

Next, consider the objective $q({x}_{0:N}, {u}_{0:N-1}) = \| x_N - z \|$ where $z \in \delta \Yset_i$ is some  point on the boundary of $\Yset_i$ with a reachable neighborhood $N(z) = \{ y \in \Yset_i : \| y - z \| \leq \epsilon \}$ for some $\epsilon$. By construction, this objective is continuous and minimized by the feasible solution $x^\star_N = z$. 

By construction, we have $q({x}_{0:N}, {u}_{0:N-1}) > q({x}_{0:N}^\star, {u}_{0:N-1}^\star)$ for any feasible sequence ${x}_{0:N}$ with terminal condition $x_N \neq z \in N(z)$. In contrast, $\hat{q}({x}_{0:N}, {u}_{0:N-1}) \leq \hat{q}({x}_{0:N}^\star, {u}_{0:N-1}^\star) = \alpha_i$ for any terminal condition $x_N \in N(z)$. Thus, there exists $x_N \in N(z)$ that is optimal with respect to $\hat{q}$, $\hat{q}({x}_{0:N}, {u}_{0:N-1}) \leq \hat{q}({x}_{0:N}^\star, {u}_{0:N-1}^\star)$, but is sub-optimal with respect to $q$, \ie~$q({x}_{0:N}, {u}_{0:N-1}) > q({x}_{0:N}^\star, {u}_{0:N-1}^\star)$. 
\end{proof}

As a consequence of Corollary~\ref{cor:optimal_blameless}, it is possible for the user to define a continuous mission objective $q$ for which it is impossible to formulate a single continuous optimal control problem that produces control sequences that are both blameless and optimal. In other words, Problem~\ref{prob:cts_obj} does not necessarily have a solution.

It follows that a blamelessly optimal control sequence must be found by solving at least two sub-problems. We propose an algorithm for finding blamelessly optimal control sequences using exactly two sub-problems, rather than up to $m$ problems like in Algorithm~\ref{alg:brute}. The proposed algorithm first finds the highest priority set for which a feasible solution exists. Then, the optimal control sequence with respect to the mission objective subject to the highest priority constraint is found. The algorithm is presented in Algorithm~\ref{alg:blameless}. For further details on how $i^\star$ is found in Algorithm~\ref{alg:blameless}, see~\cite{github}.

\begin{algorithm}
\caption{Two-Stage Blameless Control}
\label{alg:blameless}
\begin{algorithmic}
\State Minimize $i^\star$ such that $\Yset_{i^\star} \cap \mathrm{Suc}_N(\Xset_0) \neq \varnothing$
\State Solve \vspace{-1em}
\begin{subequations}
\label{eq:min-blameless}
\begin{alignat}{3}
&& \min_{u_{0:N-1}} ~& q(x_{0:N},u_{0:N-1}) \\ 
&& \st ~& (x_{0:N}, u_{0:N-1}) \in \mc{F}(\mc{X}_0, \mc{U}), \\
&&& x_{N } \in \Yset_{i^\star} \label{eq:blamless_safety}
\end{alignat}
\end{subequations}
\end{algorithmic}
\end{algorithm}

\begin{proposition}
The solution to Algorithm~\ref{alg:blameless} solves Problem~\ref{prob:Pi}.
\end{proposition}

\begin{proof}
We will prove that the solution to Algorithm~\ref{alg:blameless} is blameless directly. Assume the solution to Algorithm~\ref{alg:blameless} is $u^\star_{0:N-1}$. Then, the resulting state sequence $x^\star_{0:N}$ has $x_N \in \Yset_{i^\star}$, so that $\Yset_{i^\star}\cap \mathrm{Suc}_N(\Xset_0)\neq \varnothing$. It follows directly from Proposition~\ref{prop:reach} that the control sequence is blameless.

We will prove by contradiction that the solution to Algorithm~\ref{alg:blameless} is blamelessly optimal. Assume the solution to Algorithm~\ref{alg:blameless} is $(x^\star_{0:N}, u^\star_{0:N-1})$, but there exists a blameless solution $(x^\dag_{0:N}, u^\dag_{0:N-1})$ with $q(x^\dag_{0:N}, u^\dag_{0:N-1}) < q(x^\star_{0:N}, u^\star_{0:N-1})$. Then, $(x^\star_{0:N}, u^\star_{0:N-1})$ is sub-optimal with respect to Problem~\eqref{eq:min-blameless}, contradicting the assumption that $(x^\star_{0:N}, u^\star_{0:N-1})$ is the solution to Algorithm~\ref{alg:blameless}. Thus, the optimal solution to Algorithm~\ref{alg:blameless} is blamelessly optimal.
\end{proof}

Algorithm~\ref{alg:blameless} takes advantage of the fact that the prioritized sets are nested to find blamelessly optimal control sequences. Specifically, since high priority sets are contained in low priority sets, the problem of imposing priority within an optimal control algorithm can be posed as two decoupled problems: a set inclusion problem that finds the highest priority set that has a nonempty intersection with the dynamically feasible set, and an optimal control problem that finds the optimal control sequence that is in the set found in the first problem. The result is an algorithm that requires solving two optimization problems, rather than the $m$ required by lexicographic optimization. This becomes computational beneficial when the number of prioritized sets, $m$, is large.



%% file: results.tex
\section{Landing under Prioritized Sets} 
\label{sec:results}
We consider the problem of an autonomous lander selecting a landing site, subject to control limits and limited power. Prioritized sets are defined to dictate the most desirable landing sites. This problem is illustrated in Figure~\ref{fig:helicopter} in three-dimensional space. For ease of presentation, the results consider the problem restriction to two-dimensional space, so that the landing sites are restricted to a line.

\subsubsection*{Lander Dynamics.}
We model the lander as a linear system. The state consists of the velocity and position of the lander, $x = \bma{cccc} \dot{r}^x & \dot{r}^y & r^x & r^y \ema^\trans$. The control input is acceleration $u = \bma{cc} a^x & a^y \ema^\trans$. The affine continuous-time dynamics are $\dot{x} = A x + B u + C g$
%
%
where 
\begin{align*}
    A = \bma{cc}
        \mbf{0}_{2\times2} & \mbf{0}_{2\times2}\\
        I_{2\times2} & \mbf{0}_{2\times2}
    \ema,\quad 
    B = \bma{c}
        I_{2\times2} \\
        \mbf{0}_{2\times2}
    \ema,\quad
    C = \bma{c}
        0\\
        1\\
        0_{2\times1}
    \ema,
\end{align*}
and $g = 9.81 \frac{{\text{m}}}{{\text{s}}^2}$ is the acceleration due to gravity. The control is subject to the box constraints
\begin{align*}
    a^x \in \big[-10, 10 \big] \frac{{\text{m}}}{{\text{s}}^2} \text{ and } a^y \in \big[9,30\big] \frac{{\text{m}}}{{\text{s}}^2}.
\end{align*}
%
The lander has initial condition
\begin{align*}
    {x}_0 = \bma{cccc}-10 \frac{{\text{m}}}{{\text{s}}}& -5 \frac{{\text{m}}}{{\text{s}}} & -130{\text{m}}& 100{\text{m}} \ema^\trans,
\end{align*}
and sufficient power to last time $T = 12$ seconds.

\subsubsection*{Prioritized constraints.}
Five prioritized safety constraints are imposed on the states $\dot{r}_N^x$ and ${r}_N^x$, jointly denoted $x_N^x = \bma{cc}\dot{r}_N^x & {r}_N^x \ema$. Namely, the position at which the lander lands and the velocity in the direction parallel to the ground at the time of landing are constrained by the prioritized safety constraints. The resulting nested constraint sets are shown in Figure~\ref{fig:Ysets} and defined as follows
\begin{align*}
    \Yset_1 &= \{(\dot{r}^x_N, {r}^x_N) \mid \dot{r}^x_N \in \left[-0.5, 0.5\right] \frac{{\text{m}}}{{\text{s}}}, {r}^x_N \in \left[ -5,5 \right]{\text{m}} \},\\
    \Yset_2 &= \{(\dot{r}^x_N, {r}^x_N) \mid \dot{r}^x_N \in \left[-4, 4\right] \frac{{\text{m}}}{{\text{s}}}, {r}^x_N \in \left[ -15,12 \right]{\text{m}} \},\\
    \Yset_3 &= \{(\dot{r}^x_N, {r}^x_N) \mid \dot{r}^x_N \in \left[-7, 7\right] \frac{{\text{m}}}{{\text{s}}}, {r}^x_N \in \left[ -30,26 \right]{\text{m}} \},\\
    \Yset_4 &= \{(\dot{r}^x_N, {r}^x_N) \mid \dot{r}^x_N \in \left[-10, 10\right] \frac{{\text{m}}}{{\text{s}}}, {r}^x_N \in \left[ -40,35 \right]{\text{m}} \},\\
    \Yset_5 &= \{(\dot{r}^x_N, {r}^x_N) \mid \dot{r}^x_N \in \left[-15, 15\right] \frac{{\text{m}}}{{\text{s}}}, {r}^x_N \in \left[ -45,52 \right]{\text{m}} \}.
\end{align*}
\begin{figure}
    \centering
        \includegraphics[width=0.6\columnwidth, trim={4cm 3cm 3cm 2cm},clip]{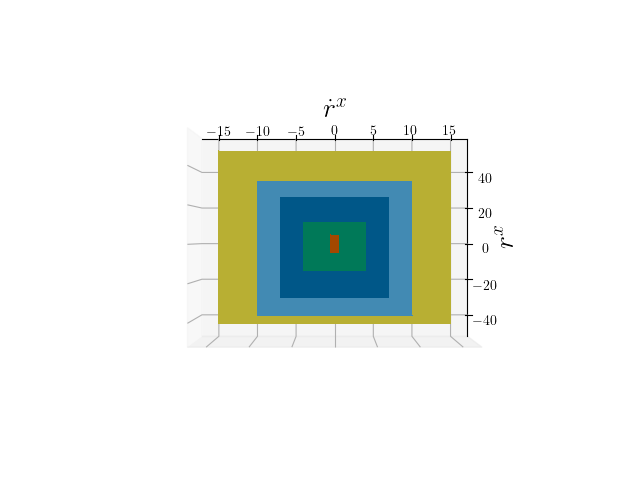}
        \caption{Sets $\Yset_1,\ldots,\Yset_5$.}
        \label{fig:Ysets}
\end{figure}

\subsubsection*{Objective Function.} 
The objective function~\eqref{eq:objective} is the quadratic objective $q(x_{0:N}, u_{0:N-1}) = \sum_{k=1}^N q_k(x_{k}, u_{k-1})$, with
\begin{align} \label{eq:results_obj}
    q_k(x_{k}, u_{k-1}) = u_{k-1}^\trans R u_{k-1} \!+\! (x^x_{N} \!-\! c_i)^\trans Q (x^x_{N}\!-\!c_i),
\end{align}
where $R\in\reals^{2\times2}$ is the input cost matrix, $Q\in\reals^{2\times 2}$ is the regulator cost matrix, and $c_i \in \reals^2$ is the center of the safety set $\Yset_i$. The weights, $Q$ and $R$ are tuned by the user.

Blamelessly optimal control sequences are found by solving Algorithm~\ref{alg:blameless}. First, an objective function is generated that satisfies the necessary and sufficient conditions for an objective function that produces blameless, but not necessarily optimal, control sequences, according to Theorem~\ref{thm:optimal_blameless}. This objective is used to determine the smallest index $i^\star$ such that $x_N^x \in \Yset_{i^\star}$. Then, the control sequence $u^\star_{0:N-1}$ that is optimal with respect to~\eqref{eq:results_obj} subject to dynamics~\eqref{eqn:opt_ctrl_feasible_set} and terminal state constraint $x^x_N \in \Yset_{i^{\star}}$ is found. 

The results of Algorithm~\ref{alg:blameless} are compared to the results of Algorithm~\ref{alg:brute}, and an optimal control algorithm that is optimal with respect to the quadratic objective $q(x_{0:N}, u_{0:N-1}) = \sum_{k=1}^N q_k(x_{k}, u_{k-1})$, with
\begin{align}\label{eq:objective_regulator}
    q_k(x_{k}, u_{k-1}) = u_{k-1}^\trans R u_{k-1} \!+\! (x^x_{N} \!-\! c_1)^\trans Q (x^x_{N} \!-\! c_1),
\end{align}
where $c_1$ is the center of the highest priority safety set. This algorithm does not guarantee that the resulting control sequences are blameless. 

The weighting matrices are $R = \diag(0.5, 1)^2 \frac{1}{{\text{m}}^2}$ and $Q = 5^2 I \frac{{\text{s}}^2}{{\text{m}}^2}$ for the blamelessly optimal and brute force algorithms. The optimal algorithm uses $R = \diag(0.5, 1)^2 \frac{1}{{\text{m}}^2}$ and is tested with two values of the input weighting matrix, $Q = 5^2I \frac{{\text{s}}^2}{{\text{m}}^2}$ and $Q = 0.15^2 I \frac{{\text{s}}^2}{{\text{m}}^2}$.

\begin{figure}
     \centering
     \includegraphics[width=0.95\columnwidth]{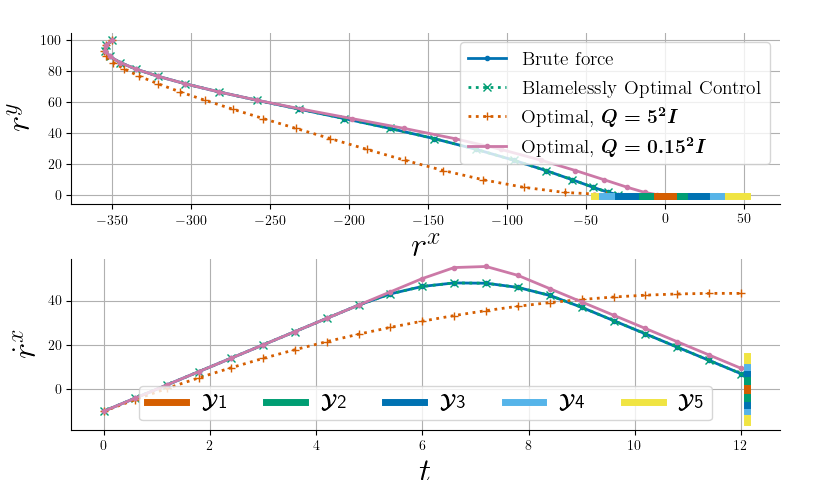}
     \caption{Comparison of blameless optimality, brute force and optimal control algorithm for 12 second trajectory.}\vspace{-0.6cm}
     \label{fig:alg_comp}
\end{figure}

\subsubsection*{Results} Figure~\ref{fig:alg_comp} shows the trajectories and velocity in the $x$-direction generated using the brute force algorithm (Algorithm~\ref{alg:brute}), the blamelessly optimal control algorithm (Algorithm~\ref{alg:blameless}), and the algorithm optimal with respect to~\eqref{eq:objective_regulator} with two 
weights. 
The brute force and blamelessly optimal trajectories are identical, and both result in solutions that have terminal states in the set $\Yset_3$. The algorithm optimal with respect to~\eqref{eq:objective_regulator} with $Q = 5^2I \frac{{\text{s}}^2}{{\text{m}}^2}$ results in a terminal position in $\Yset_4$ and a terminal velocity outside of the safety sets, and is therefore ultimately outside of the safety sets. The trajectory resulting from the same problem with input cost matrix $Q = 0.15^2 I \frac{{\text{s}}^2}{{\text{m}}^2}$ has a terminal position in $\Yset_2$ and a terminal velocity in $\Yset_4$, and therefore ultimately a terminal state in $\Yset_4$.

\subsection{Implications of Blameless Optimal Control}
The brute force algorithm (Algorithm~\ref{alg:brute}), and the blamelessly optimal algorithm (Algorithm~\ref{alg:blameless}) result in equivalent solutions. However, the brute force algorithm requires solving up to $m$ optimization problems given $m$ safety constraints. The blameless optimization problem requires solving two optimization problems, regardless of the number of defined prioritized safety constraints. The blamelessly optimal algorithm is therefore increasingly beneficial when several safety constraints are defined.

In safety critical applications, guarantees on the blamelessness of a control sequence are required to enable trust of autonomous systems. The algorithm optimal with respect to objective~\eqref{eq:objective_regulator} requires parameter tuning to find the safest possible solution. 
Moreover, the difficulty of tuning the input weighting matrix and regulator weighting matrix increases with the number of states being constrained. When problem parameters such as initial conditions and problem horizon are uncertain, tuning the weighting matrices is impractical and could lead to blameworthy control sequences.

%% file: conclusion.tex
\section{Conclusions}\label{conclusion}

This work develops the concept of a blamelessly optimal control sequence. We show that there is a tradeoff between optimality and blamelessness, which motivates the need for blameless optimality. An algorithm is presented to solve for blamelessly optimal control sequences and results are presented on a rocket landing problem. Future work will expand the idea of blameless optimality to stochastic systems.